\documentclass[11pt]{amsart}
\usepackage{amscd}
\usepackage{amsmath,empheq}
\usepackage{amsfonts}
\usepackage{amssymb}
\usepackage{mathrsfs}

\usepackage[all]{xy}

\usepackage{mathtools}

\newtheorem{theorem}{Theorem}

\newtheorem{lemma}[theorem]{Lemma}
\newtheorem{proposition}[theorem]{Proposition}
\newtheorem{remark}{Remark}

\newenvironment{proof-sketch}{\noindent{\bf Sketch of Proof}\hspace*{1em}}{\qed\bigskip}

\everymath{\displaystyle}

\newcommand{\RR}{\mathbb R}

\newcommand{\ZZ}{\mathbb Z}
\newcommand{\di}{\displaystyle}

\renewcommand{\leq}{\leqslant}

\renewcommand{\geq}{\geqslant}

\begin{document}
%\hfill\today\bigskip

\title[The Lane-Emden equation with variable double-phase and multiple regime]{The Lane-Emden equation with variable double-phase and multiple regime}

%%%%%%%%%%%%%%%%%%%%%%%%%%%%%%%%%%%%%%%%%%%%%%%%%%%%%%%%%%%%%%%%%%%%%%%
\author[C.O. Alves]{Claudianor O. Alves}
\address[C.O. Alves]{Unidade Acad\^emica de Matem\'atica, Universidade Federal de Campina Grande, 58429-970, Campina Grande - PB, Brazil}
\email{\tt coalves@mat.ufcg.edu.br}

\author[V.D. R\u{a}dulescu]{Vicen\c{t}iu D. R\u{a}dulescu}
\address[V.D. R\u{a}dulescu]{Faculty of Applied Mathematics, AGH University of Science and Technology, 30-059 Krak\'ow, Poland \& Department of Mathematics, University of Craiova, 200585 Craiova, Romania }
\email{\tt radulescu@inf.ucv.ro}

\noindent\keywords{Lane-Emden equation, double-phase energy, variable exponent, nonhomogeneous differential operator, Palais principle of symmetric criticality.
\\
{\it 2010 Mathematics Subject Classification}. Primary: 35J20. Secondary: 35J75, 35J92, 35P30.}

\begin{abstract}
We are concerned with the study of the Lane-Emden equation with variable exponent and Dirichlet boundary condition. The feature of this paper is that the analysis that we develop does not assume any subcritical hypotheses and the reaction can fulfill a mixed regime (subcritical, critical and supercritical). We consider the radial and the nonradial cases, as well as a singular setting. The proofs combine variational and analytic methods with a version of the Palais principle of symmetric criticality.
\end{abstract}

\maketitle

\section{Introduction and abstract setting}

Let $\Omega\subseteq\RR^N$ be a bounded regular connected open set and assume that $p,q\in (1,\infty)$. The Lane-Emden problem
\begin{equation}\label{LE}\left\{\begin{array}{lll}
&\di -\Delta_pu=|u|^{q-2}u&\quad\mbox{in}\ \Omega\\
&\di u=0&\quad\mbox{on}\ \partial\Omega\\
&\di u\not\equiv 0&\quad\mbox{in}\ \Omega\end{array}\right.
\end{equation}
and its numerous versions play a central role in the qualitative analysis of nonlinear elliptic equations. Usually, this analysis is developed in relationship with the values of $q$ with respect to the Sobolev critical exponent $p^*$ of $p$, which is defined by
$$p^*=\left\{\begin{array}{lll}
&\di \frac{Np}{N-p}&\quad\mbox{if}\ 1<p<N\\
&+\infty&\quad\mbox{if}\ p\geq N.\end{array}\right.
$$

The following three basic situations can occur:

\smallskip (i) $q<p^*$ ({\it subcritical case}). Then the associated energy functional is either coercive (if $q<p$) or has a mountain pass geometry and satisfies the Palais-Smale condition (if $q>p$), hence problem \eqref{LE} has at least one solution. The case $p=q$ corresponds to an eigenvalue problem, so we cannot exclude a nonexistence property.

\smallskip (ii) $q=p^*$, provided that $1<p<N$ ({\it critical case}). In this case, the topology of $\Omega$ plays a crucial role. We recall the following basic result of Bahri and Coron \cite{bahri}. Let $H_d(\Omega;\ZZ_2)$ denote the homology of dimension $d$ of $\Omega$ with $\ZZ_2$-coefficients. The Bahri-Coron theorem states that if there exists a positive integer $d$ such that $H_d(\Omega;\ZZ_2)\not=0$, then problem \eqref{LE} has at least one positive solution. In particular, if $p=2$, $N=3$, $q=6$ and $\Omega$ is not contractible, then problem \eqref{LE} has at least one positive solution.

\smallskip (iii) $q>p^*$, provided that $1<p<N$ ({\it supercritical case}). This situation is delicate and a major role is played by the geometry of $\Omega$. For instance, if $\Omega$ is starshaped then problem \eqref{LE} does not have any solution (by Pohozaev's identity). Also, if $\Omega$ is an annulus, problem \eqref{LE} always has at least one solution.

In the case of variable exponents, the Lane-Emden problem \eqref{LE} becomes
\begin{equation}\label{LEvar}\left\{\begin{array}{lll}
&\di -\Delta_{p(x)}u=|u|^{q(x)-2}u&\quad\mbox{in}\ \Omega\\
&\di u=0&\quad\mbox{on}\ \partial\Omega\\
&\di u\not\equiv 0&\quad\mbox{in}\ \Omega,\end{array}\right.
\end{equation}
where $\Delta_{p(x)}u:={\rm div}\, (|\nabla u|^{p(x)-2}\nabla u)$. In this case, the critical exponent of $p(x)$ depends on the point and it is defined by
$$p^*(x)=\left\{\begin{array}{lll}
&\di \frac{Np(x)}{N-p(x)}&\quad\mbox{if}\ 1<p(x)<N\\
&+\infty&\quad\mbox{if}\ p(x)\geq N.\end{array}\right.
$$

Problem \eqref{LEvar} and its versions have been intensively studied in several recent works, in relationship with numerous applications to non-Newtonian fluids, image reconstruction, etc. We refer to R\u adulescu and Repov\v{s} \cite{radrep} for a comprehensive study of several classes of nonlinear problems with variable exponent. For instance, in \cite{radrep} it is studied the problem
\begin{equation}\label{LE1}\left\{\begin{array}{lll}
&\di -\Delta_{p(x)}u=\lambda\,|u|^{q(x)-2}u&\quad\mbox{in}\ \Omega\\
&\di u=0&\quad\mbox{on}\ \partial\Omega\\
&\di u\not\equiv 0&\quad\mbox{in}\ \Omega\end{array}\right.
\end{equation}
under the following hypotheses:\\
(h1) $1<\min_{x\in\overline\Omega}q(x)<\min_{x\in\overline\Omega}p(x)<\max_{x\in\overline\Omega}q(x)$;\\
(h2) $q(x)<p^*(x)$ for all $x\in\overline\Omega$.

 Theorem 3 in \cite[p. 43]{radrep} studies the case of low perturbations of the reaction and establishes that there exists $\lambda^*>0$ such that problem \eqref{LE1} has at least one solution  for all $\lambda \in(0,\lambda^*)$.

 Hypothesis (h2) corresponds to the {\it subcritical} case. However, the presence of {\it variable} exponents allows the following {\it almost critical} case: there is $z\in\Omega$ such that $q(x)<p^*(x)$ for all $x\in\Omega\setminus\{z\}$ and $q(z)=p(z)$. In this direction, we refer to Alves, Ercole and Huam\'an Bola\~nos \cite{AED}, where  the problem \ref{LE1} has been considered with $p(x)=2$ and $q(x)$  a functional that has a critical growth in a subset $ \Omega$.
 By using variational methods combined with the Lions concentration-compactness principle \cite{Lions}, the authors proved the existence of ground state solutions.

 Furthermore, problem \eqref{LEvar} can fulfill even a ``subcritical-critical-supercritical" triple regime, in the sense that $\Omega=\Omega_1\cup\Omega_2\cup\Omega_3$ and
 $$q(x)<p^*(x)\quad\mbox{if $x\in\Omega_1$};$$
 $$q(x)=p^*(x)\quad\mbox{if $x\in\Omega_2$};$$
 $$q(x)>p^*(x)\quad\mbox{if $x\in\Omega_3$}.$$

To the best of our knowledge, nonlinear PDEs with variable exponent and multiple regime (subcritical, critical and supercritical) have not been considered in the literature. The present paper is the first study dedicated to the qualitative analysis of the Lane-Emden equation with variable exponent and multiple regime.
We consider the radial and the nonradial cases, as well as a singular setting. More precisely,
in Section 2 we study the case where $\Omega$ is a ball centered at origin. In this abstract framework, a key role in the proofs is played by a recent version of the Palais principle of symmetric criticality developed by Kobayashi and Otani \cite[Theorem 2.2]{KS}. Let $E$ be a Banach space on which a symmetric group $G$ acts and let $J$ be a $G$-invariant functional defined on $E$. We recall that Palais \cite{palais} gave some sufficient conditions to guarantee the principle of symmetric criticality, in the sense that every critical point of $J$ restricted to the subspace of $G$-symmetric points becomes also a critical point of $J$ on the whole space $E$. This principle has been generalized by Kobayashi and Otani \cite{KS} to the case where $J$ is not differentiable within the setting that does not require the full variational structure, under the hypothesis that the action of $G$ is isometry or $G$ is compact. In Sections 3 and 4, we study the case where $\Omega$ is not a ball centered at origin, which is more delicate because we cannot used the principle of symmetric criticality mentioned above. 
In both sections we are concerned with the existence of nontrivial  solutions in the nonradial case. We overcome the lack of symmetry properties by developing a new strategy, see proof of Theorem \ref{t2} for more details. In the last section of the present paper we are concerned with a singular setting, which corresponds to a vanishing potential.

\subsection{Variable exponent Lebesgue and Sobolev spaces}
In this subsection, we recall some results on variable exponent Lebesgue and Sobolev spaces. For more details we refer to \cite{die,FSZ,radrep} and their references.

Let $p\in L^{\infty}(\mathbb{R}^{N})$ with $ p_-:={\rm essinf}_{x\in\RR^N}\,p(x)> 1$. The variable exponent Lebesgue space
$L^{p(x)}(\mathbb{R}^{N})$ is defined by
\[
L^{p(x)}(\mathbb{R}^{N})=\left\{
u:\mathbb{R}^{N} \to  \mathbb{R} \left\vert \,u\text{ is
	measurable and }\int_{\mathbb{R}^{N}}\left\vert u\right\vert
^{p(x)}\,dx<\infty\right.  \right\}
\]
endowed with the norm
\[
\left\vert u\right\vert _{p(x)}=\inf\left\{  \lambda>0\left\vert
\,\int_{\mathbb{R}^{N}}\left\vert \frac{u}{\lambda}\right\vert
^{p(x)}\,dx \leq1\right.  \right\}  \text{.}
\]
The variable exponent Sobolev space is defined by
\[
W^{1,p(x)}(\mathbb{R}^{N})=\left\{  u\in
L^{p(x)}(\mathbb{R}^{N})\left\vert \,\left\vert \nabla u\right\vert
\in L^{p(x)}(\mathbb{R}^{N})\right. \right\}
\]
with the norm
\[
\left\Vert u\right\Vert _{1,p(x)}=\left\vert u\right\vert
_{p(x)}+\left\vert \nabla u\right\vert _{p(x)}\text{.}
\]
With these norms, the spaces $L^{p(x)}(\mathbb{R}^{N})$ and
$W^{1,p(x)}(\mathbb{R}^{N})$ are reflexive and separable Banach spaces.

\begin{proposition}
	\label{p1}The functional $\rho:W^{1,p(x)}(\mathbb{R}^{N}) \to \mathbb{R}$ defined by
	\begin{equation}
	\rho(u)=\int_{\mathbb{R}^{N}}\left(  \left\vert \nabla u\right\vert
	^{p(x)}+\left\vert u\right\vert ^{p(x)}\right)\,dx
	\text{,} \label{ps}
	\end{equation}
	has the following properties:
	
	\begin{enumerate}
		\item[\emph{(i)}] If $\left\Vert u\right\Vert_{1,p(x)} \geq1$, then $\left\Vert u\right\Vert_{1,p(x)}
		^{p_{-}}\leq\rho(u)\leq\left\Vert u\right\Vert_{1,p(x)} ^{p_+}$.
		
		\item[\emph{(ii)}] If $\| u\|_{1,p(x)} \leq1$, then $\left\Vert u\right\Vert_{1,p(x)}
		^{p_+}\leq\rho(u)\leq\left\Vert u\right\Vert_{1,p(x)} ^{p_{-}}$.
	\end{enumerate}
	In particular, $\rho(u)=1$ if and only if $\left\Vert u \right\Vert_{1,p(x)} =1$ and if $ (u_n) \subset W^{ 1,p(x) }( \mathbb R^N ) $, then $\left\Vert u_{n}\right\Vert_{1,p(x)}
	\rightarrow0$ if and only if $\rho( u_{n}) \rightarrow0$.
\end{proposition}

\begin{remark}
	\label{r1}For the functional
	$\xi:L^{p(x)}(\mathbb{R}^{N})\rightarrow \mathbb{R}$ given by
	\[
	\xi(u)=\int_{\mathbb{R}^{N}}\left\vert u\right\vert
	^{p(x)}\,dx\text{,}
	\]
	the conclusion of Proposition \ref{p1} also holds, for example, if $ (u_n) \subset L^{p(x) }( \mathbb R^N ) $, then  $\left\vert u_{n}\right\vert
	_{p(x)}\rightarrow0$ if and only if $\xi(u_{n})\rightarrow0$. Moreover, from $(i)$ and $(ii)$,
	\begin{equation}
	\left\vert u\right\vert _{p(x)}\leq \max\left\{
	\left(\int_{\mathbb{R}^{N}}\left\vert u\right\vert
	^{p(x)}\,dx\right)  ^{1/p_{-}},\left(  \int_{\mathbb{R}^{N}
	}\left\vert u\right\vert ^{p(x)}\,dx\right)
	^{1/p_+} \right\} \text{.} \label{in}
	\end{equation}
\end{remark}

Related to the Lebesgue space $L^{h(x)}(\mathbb{R}^{N})$, we have the following generalized H\"{o}lder-type inequality.

\begin{proposition}
	[{\cite[p. 9]{Mu}}]\label{h} For $p\in L^{\infty}( \mathbb{R}^{N}) $ with $ p_- >1$, let  $p^{\prime }:\mathbb{R}^{N} \to \mathbb{R}$ be such that
	\[
	\frac{1}{p( x) }+\frac{1}{p^{\prime}( x) }=1\text{,\qquad a.e. }x\in
	\mathbb{R}^{N}\text{.}
	\]
	Then, for any $u\in L^{p( x) }( \mathbb{R}^{N})
	$ and  $v\in L^{p^{\prime}( x) }( \mathbb{R}^{N}) $,
	\begin{equation}
	\left\vert \int_{\mathbb{R}^{N}}uv\,dx\,\right\vert
	\leq \left( \frac{1}{p_-} + \frac{1}{p'_-} \right)\left\vert u\right\vert _{p( x) }\left\vert v\right\vert
	_{p^{\prime}( x) }\text{.} \label{hi}
	\end{equation}
\end{proposition}

\begin{proposition}
	[{\cite[Theorems 1.1, 1.3]{FSZ}}]\label{pp1}Let $p:\mathbb{R}^{N}
	\to \mathbb{R}$ be a Lipschitz continuous satisfying $ 1 < p_- \leq p_+ < N $ and $t:\mathbb{R}^{N} \rightarrow \mathbb{R}$ be a measurable function.
	
	\begin{enumerate}
		\item[\emph{(i)}] If $p\leq t\leq p^{\ast}$, the embedding
		$W^{1,p(x)}(\mathbb{R}^{N})\hookrightarrow
		L^{t(x)}(\mathbb{R}^{N})$ is continuous.
		
		\item[\emph{(ii)}] If $p\leq t\ll p^{\ast}$, the embedding
		$W^{1,p(x)}(\mathbb{R}^{N})\hookrightarrow L_{\mathrm{loc}}^{t(x)}
		(\mathbb{R}^{N})$  is compact.
	\end{enumerate}
\end{proposition}

The  Lebesgue and Sobolev
spaces with variable exponents coincide with the usual Lebes\-gue and Sobolev spaces provided that $p$ is constant. According to \cite[pp. 8-9]{radrep}, these function spaces $L^{p(x)}(\mathbb{R}^N)$ and $W^{1,p(x)}(\mathbb{R}^N)$ have some non-usual properties,
such as:

(i)
Assuming that $1<p^-\leq p^+<\infty$ and $p:\overline\Omega\rightarrow [1,\infty)$ is a smooth function, then the following co-area formula
$$\int_\Omega |u(x)|^pdx=p\int_0^\infty t^{p-1}\,|\{x\in\Omega ;\ |u(x)|>t\}|\,dt$$
has  no analogue in the framework of variable exponents.

(ii) Spaces $L^{p(x)}(\mathbb{R}^N)$ do {\it not} satisfy the {\it mean continuity property}. More exactly, if $p$ is nonconstant and continuous in an open ball $B$, then there is some $u\in L^{p(x)}(B)$ such that $u(x+h)\not\in L^{p(x)}(B)$ for every $h\in{\mathbb R}^N$ with arbitrary small norm.

(iii) Function spaces with variable exponent
 are {\it never} invariant with respect to translations.  The convolution is also limited. For instance,  the classical Young inequality
$$| f*g|_{p(x)}\leq C\, | f|_{p(x)}\, \| g\|_{L^1}$$
remains true if and only if
$p$ is constant.

\subsection{Double-phase problems and their historical traces}

Let $\Omega$ be a bounded domain in $\RR^N$ ($N\geq 2$) with a smooth boundary. If $u:\Omega\to\RR^N$ is the displacement and $Du$ is the $N\times N$  matrix of the deformation gradient, then the total energy can be represented by an integral of the type
\begin{equation}\label{paolo}I(u)=\int_\Omega f(x,Du(x))dx,\end{equation}
where the energy function $f=f(x,\xi):\Omega\times\RR^{N\times N}\to\RR$ is quasiconvex with respect to $\xi$, see Morrey \cite{morrey}. One of the simplest examples considered by Ball is given by functions $f$ of the type
$$f(\xi)=g(\xi)+h({\rm det}\,\xi),$$
where ${\rm det}\,\xi$ is the determinant of the $N\times N$ matrix $\xi$, and $g$, $h$ are nonnegative convex functions, which satisfy the growth conditions
$$g(\xi)\geq c_1\,|\xi|^p;\quad\lim_{t\to+\infty}h(t)=+\infty,$$
where $c_1$ is a positive constant and $1<p<N$. The condition $p\leq N$ is necessary to study the existence of equilibrium solutions with cavities, that is, minima of the integral \eqref{paolo} that are discontinuous at one point where a cavity forms; in fact, every $u$ with finite energy belongs to the Sobolev space $W^{1,p}(\Omega,\RR^N)$, and thus it is a continuous function if $p>N$. In accordance with these problems arising in nonlinear elasticity, Marcellini \cite{marce1,marce2,marce3} considered continuous functions $f=f(x,u)$ with {\it unbalanced growth} that satisfy
$$c_1\,|u|^p\leq |f(x,u)|\leq c_2\,(1+|u|^q)\quad\mbox{for all}\ (x,u)\in\Omega\times\RR,$$
where $c_1$, $c_2$ are positive constants and $1\leq p\leq q$. Regularity and existence of solutions of elliptic equations with $p,q$--growth conditions were studied in \cite{marce2}.

The study of non-autonomous functionals characterized by the fact that the energy density changes its ellipticity and growth properties according to the point has been continued by Mingione {\it et al.} \cite{1Bar-Col-Min, 2Bar-Col-Min, beck, 8Col-Min, 9Col-Min}, R\u adulescu {\it et al.} \cite{bahr19, cencelj, prrzamp, prrpams, zhang}, etc. These contributions are in relationship with the work of Zhikov \cite{23Zhikov, 24Zhikov}, which describe the
behavior of phenomena arising in nonlinear
elasticity. In fact, variational problems with nonstandard integrands were introduced at the beginning of the
1980's and were studied in the context of averaging and the Lavrent'ev phenomenon.
 Zhikov  provided models for strongly anisotropic materials in the context of homogenisation.
In particular, he considered the following model
functional
\begin{equation}\label{mingfunc}
{\mathcal P}_{p,q}(u) :=\int_\Omega (|Du|^p+a(x)|Du|^q)dx,\quad 0\leq a(x)\leq L,\ 1<p<q,
\end{equation}
where the modulating coefficient $a(x)$ dictates the geometry of the composite made of
two differential materials, with hardening exponents $p$ and $q$, respectively.

Another significant model example of a functional with $(p,q)$--growth studied by Mingione {\it et al.} is given by
$$u\mapsto \int_\Omega |Du|^p\log (1+|Du|)dx,\quad p\geq 1,$$
which is a logarithmic perturbation of the $p$-Dirichlet energy.

\section{Problem 1: the radial case}

In this section we consider $p,q,m,a:\overline{B}_R(0) \to \mathbb{R}$ four continuous functions satisfying : \\
$$
\left\{
\begin{array}{l}
1<p_-=\min_{x \in \overline{B}_R(0) }p(x) \leq \max_{x \in \overline{B}_R(0) }p(x)=p_+<N.\\
1<m_-=\min_{x \in \overline{B}_R(0) }m(x) \leq \max_{x \in \overline{B}_R(0) }m(x)=m_+<N.
\end{array}
\right.
\eqno(H1)
$$
$$
0 \leq a(x) \leq L, \quad \forall x \in \overline{B}_R(0). \eqno(H2)
$$
$$
p(x)=p(|x|), a(x)=a(|x|) \quad \mbox{and} \quad q(x)=q(|x|), \quad \forall x \in \overline{B}_R(0).
\eqno(H3)
$$
We assume that there exists  $0<r<R$ such that
$$
q(x) \geq 0 \quad \forall x \in \overline{\Omega} \quad \mbox{and} \quad p_+<q^r_{-}=\min_{x \in \overline{B}_r(0) }q(x)\leq \max_{x \in \overline{B}_r(0) }q(x)=q^r_{+}< \min_{x\in\overline{\Omega}}p^*(x).
\eqno(H4)
$$

Note that $q$ is subcritical in  $\overline{B}_r(0)$, but there is no hypotheses on the function $q$ in the annulus $A_{R,r}= \overline{B}_R(0) \setminus B_r(0)$, hence $q$ can have a supercritical growth close to the boundary. However, note that for any $t \in (0,R)$ we have the  continuous embedding
$$
W^{1,p(x)}(B_R(0)) \hookrightarrow W^{1,p_-}(A_{R,t})
$$
and the compact embedding
$$
W_{rad}^{1,p_-}(A_{R,t}) \hookrightarrow C(\overline{A}_{R,t}),
$$
which is due to Strauss \cite{Strauss}.

Therefore the embedding
\begin{equation} \label{EMB1}
W_{rad}^{1,p(x)}(B_R(0)) \hookrightarrow C(\overline{A}_{R,t}),
\end{equation}
is compact, where
$$
W_{rad}^{1,p(x)}(B_R(0))=\{u \in W^{1,p(x)}(B_R(0))\,\;\,u(x)=u(|x|) \quad \mbox{a.e} \quad x \in B_R(0)\}.
$$
Hence, it follows that the embedding
\begin{equation} \label{EB1}
W_{rad}^{1,p(x)}(B_R(0)) \hookrightarrow L^{q(x)}(B_R(0)),
\end{equation}
is also compact, which is crucial in our approach.

In what follows, $\Delta_{p(x)}$ denotes the well known $p(x)$-Laplacian operator  and  $\Delta_{m(x),a(x)}$ is the differential operator defined by
$$
\Delta_{m(x),a(x)}u={\rm div}\,(a(x)|\nabla u|^{m(x)-2}\nabla u).
$$
Moreover, when $a \not=0$, we set
$$
E=W_0^{1,p(x)}(B_R(0)) \cap W_{a(x),0}^{1,m(x)}(B_R(0)),
$$
 where $W_{a(x),0}^{1,m(x)}(B_R(0))$ is the usual space $W_{0}^{1,m(x)}(B_R(0))$,  endowed with the norm
\[
\| \nabla u\| _{m(x),a(x)}=\inf\left\{  \lambda>0\left\vert
\,\int_{\mathbb{R}^{N}}a(x)\left\vert \frac{|\nabla u|}{\lambda}\right\vert
^{m(x)}\,dx \leq1\right.  \right\}.
\]
Hereafter, we endow $E$ with the norm
$$
\|u\|=\|\nabla u\|_{p(x)}+\|\nabla u\|_{m(x),a(x)}.
$$
We observe that if  $a=0$, then $E=W_0^{1,p(x)}(B_R(0))$ and $\|\,\,\,\,\|$ is exactly the usual norm in  $W_0^{1,p(x)}(B_R(0))$.

From the definition of $E$, we have the continuous embedding
$$
E \hookrightarrow  W_0^{1,p(x)}(B_R(0)).
$$
This fact combined with (\ref{EB1}) implies that the embedding
\begin{equation} \label{EB1*}
E_{rad}(B_R(0)) \hookrightarrow L^{q(x)}(B_R(0)),
\end{equation}
is also compact, where
$$
E_{rad}=W_{rad,0}^{1,p(x)}(B_R(0)) \cap W_{rad,0}^{1,m(x)}(B_R(0)).
$$

In the sequel, $B_R$ denotes the ball $B_R(0)$.

\begin{theorem}\label{t1}
Assume that conditions $(H1)-(H4)$ are fulfilled. Then the following nonhomogeneous boundary value problem
$$
\left\{
\begin{array}{l}
-\Delta_{p(x)}u-\Delta_{m(x),a(x)}u=|u|^{q(x)-2}u \quad \mbox{in} \quad B_R, \\
u=0 \quad \mbox{on} \quad \partial B_R
\end{array}
\right.
\eqno{(P_1)}
$$
has a nontrivial solution in $E$.
\end{theorem}

\begin{proof}
The natural candidate to be the energy associated to problem $(P_1)$ is  the following double-phase functional with variable exponents:
$$
I(u)=\int_{B_R}\left(\frac{1}{p(x)}|\nabla u|^{p(x)}+\frac{a(x)}{m(x)}|\nabla u|^{m(x)}\right)\,dx-\int_{B_R}\frac{1}{q(x)}|u|^{q(x)}\,dx.
$$
However, this functional is not well defined on the whole space $E$
because we do not assume any growth condition on $q$ in the annulus $A_{R,r}$. In the sequel we will restrict  $I$ to $E_{rad}$, because $I \in C^{1}(E_{rad},\mathbb{R})$ and for all $u,v \in E_{rad}$
$$
I'(u)v=\int_{B_R}(|\nabla u|^{p(x)-2}\nabla u\nabla v+a(x)|\nabla u|^{m(x)-2}\nabla u\nabla v)\,dx-\int_{B_R}|u|^{q(x)-2}uv\,dx.
$$

Now, it is easy to prove that $I$ satisfies the mountain pass geometry and also the $(PS)$ condition, because we have the compact embedding (\ref{EB1}).  From this, we can apply the mountain pass theorem to find a nontrivial critical point $u \in E_{rad}$.

Our goal is to prove that $u$ is in fact a critical point of $I$ in the whole space $E$. However, we cannot applied directly the Palais principle of symmetric criticality, because $I$ is not well defined in whole  $E$. In order to overcome this difficulty, we will use the following trick: consider the function
$$
g(x,t)=\xi(|x|)|t|^{q(x)}+(1-\xi(|x|))|u(x)|^{q(x)}, \quad \forall x \in B_R,
$$
where $\xi \in C^{\infty}([0,R],\mathbb{R})$ satisfies
$$
\xi(x)=
\left\{
\begin{array}{l}
1, \quad x \in x \in \overline{B}_{\frac{r}{2}}(0) \\
0, \quad x \in x \in \overline{B}_R(0) \setminus \overline{B}_{\frac{3r}{5}}(0).
\end{array}
\right.
$$
Since $u \in C(\overline{A}_{R,\frac{r}{2}})$ ( see (\ref{EMB1})), it follows from $(H4)$ that
$$
|g(x,t)| \leq C(|t|^{q^r_{+}}+1), \quad \forall (x,t) \in B_R \times \mathbb{R}.
$$
This fact implies that $g$ has a subcritical growth.  Consider the nonlinear problem
$$
\left\{
\begin{array}{l}
-\Delta_{p(x)}w-\Delta_{m(x),a(x)}w=g(x,w) \quad \mbox{in} \quad B_R, \\
w=0 \quad \mbox{on} \quad \partial B_R,
\end{array}
\right.
\eqno{(P_g)}
$$
whose associated energy is given by
$$
J(w)=\int_{B_R}\left(\frac{1}{p(x)}|\nabla w|^{p(x)}+\frac{a(x)}{m(x)}|\nabla w|^{m(x)}\right)\,dx-\int_{B_R}G(x,w)\,dx,
$$
where $G(x,t)=\int_{0}^{t}g(x,s)\,ds$.

Since $g$ is subcritical, it follows that $J$ is well defined in the whole space $E$, $J \in C^{1}(E,\mathbb{R})$ and
$$
J'(u)v=\int_{B_R}(|\nabla w|^{p(x)-2}\nabla w\nabla v+a(x)|\nabla w|^{m(x)-2}\nabla w\nabla v)\,dx-\int_{B_R}g(x,w)v\,dx, \quad \forall  u,v \in E.
$$
Since
$$
g(x,u(x))=|u|^{q(x)-2}u(x), \quad \forall x \in B_R,
$$
we see that $u$ is a critical point of $J$ restricted to  $E_{rad}$. Now we can apply the  Palais principle of symmetric criticality developed by Kobayashi and Otani \cite[Theorem 2.2]{KS} to conclude that $u$ is a nontrivial critical point of $J$ in the whole $E$.
\end{proof}

\section{Problem 2: the non-radial case}

In this section, we study the existence of nontrivial solutions for the following problem:
$$
\left\{
\begin{array}{l}
-\Delta_{p(x)}u-\Delta_{m(x),a(x)}u=|u|^{q(x)-2}u \quad \mbox{in} \quad \Omega, \\
u=0 \quad \mbox{on} \quad \partial \Omega,
\end{array}
\right.
\eqno{(P_2)}
$$
where $\Omega$ is a bounded domain with smooth boundary. We assume that there exist positive numbers $r<R$ such that $B_R \subset \Omega$ and
$$
 a(x)=a_0 \qquad \mbox{for all}\ x \in A_{R,r} .
$$
Related to the functions $p,q,m,a:\overline{\Omega} \to \mathbb{R}$, we assume that they are continuous and satisfy the following conditions:
$$
\left\{
\begin{array}{l}
1<p_-=\min_{x \in \overline{\Omega} }p(x) \leq \max_{x \in \overline{\Omega}}p(x)=p_+<N, \\
1<m_-=\min_{x \in \overline{\Omega} }m(x) \leq \max_{x \in \overline{\Omega} }m(x)=m_+<N.
\end{array}
\right.
\eqno(H5)
$$
$$
0 \leq a(x) \leq L, \quad \forall x \in \overline{\Omega} \eqno(H6)
$$
$$
p(x)=p(|x|) \quad \mbox{and} \quad q(x)=q(|x|), \quad \forall x \in \overline{A}_{R,r},
\eqno(H7)
$$
and
$$
q(x) \geq 0 \quad \forall x \in \overline{\Omega} \quad \mbox{and} \quad p_+<q^{A}_{-}=\min_{x \in \overline{\Omega} \setminus A_{R,r} }q(x)=q^{A}_{+}\leq \max_{x \in \overline{\Omega} \setminus A_{R,r} }q(x) < \min_{x \in \overline{\Omega}}p^*(x).
\eqno(H8)
$$

We point out that we are not assuming any growth condition on $q$ in the annulus $A_{R,r}$, hence $q$ can have a supercritical growth in that region.

\begin{theorem}\label{t2}
Assume that hypotheses $(H5)-(H8)$ are fulfilled. Then problem $(P_2)$
has a nontrivial solution in $E$.
\end{theorem}

{\it Proof.}
The  energy associated to problem $(P_2)$ is the following double-phase variational integral with variable exponents:
$$
I(u)=\int_{\Omega}\left(\frac{1}{p(x)}|\nabla u|^{p(x)}+\frac{a(x)}{m(x)}|\nabla u|^{m(x)}\right)\,dx-\int_{\Omega}\frac{1}{q(x)}|u|^{q(x)}\,dx.
$$
However, since we do not assume any growth condition on $q$ in the annulus $A_{R,r}$ $I$  is not well defined on the whole $E$. Keeping this in mind, we will restrict the function $I$ to the closed subspace $X \subset E $ given by
$$
X=\{u \in E \,:\,u(x)=u(|x|) \quad \mbox{a.e.} \quad x \in \overline{A}_{R,r} \}.
$$
Arguing as in Section 2, we observe that the compact embedding (\ref{EB1}) still holds in the present case, hence $I \in C^{1}(X,\mathbb{R})$ and
$$
I'(u)v=\int_{\Omega}(|\nabla u|^{p(x)-2}\nabla u\nabla v+a(x)|\nabla u|^{m(x)-2}\nabla u\nabla v)\,dx-\int_{\Omega}|u|^{q(x)-2}uv\,dx, \quad \forall u,v \in X.
$$
Moreover, $I$ also satisfies the $(PS)$ condition in $X$. Therefore, we can use the mountain pass theorem to get a nontrivial critical point $u_0 \in X$ of $I$, that is,
\begin{equation} \label{B0}
\int_{\Omega}(|\nabla u_0|^{p(x)-2}\nabla u\nabla v+a(x)|\nabla u_0|^{m(x)-2}\nabla u\nabla v)\,dx=\int_{\Omega}|u_0|^{q(x)-2}uv\,dx, \quad \forall v \in X.
\end{equation}

Now, we are going to show that $u_0$ is, in fact, a critical point of $I$. For this purpose  we cannot use the Palais principle used in Section 2, because $\Omega$ is not a ball.  Here, the trick is the following: for all $\varphi \in X_{0}(A_{R,r})=\{u \in X\,:\, u=0 \quad \mbox{on} \quad \partial (A_{R,r})\}$  we have
$$
\int_{A_{R,r}}(|\nabla u_0|^{p(x)-2}\nabla u\nabla v+a(x)|\nabla u_0|^{m(x)-2}\nabla u\nabla v)\,dx-\int_{A_{R,r}}|u_0|^{q(x)-2}u_0 \varphi\,dx=0.
$$

Since $p(x)=p(|x|)$, $q(x)=q(|x|)$ and $a(x)=a_0$, the regularity theory ensures that the function $$f(s)=s^{N-1}(|u'(s)|^{p(s)-2}u(s)+a_0|u'_0(s)|^{p(s)-2}u(s))$$ is of class $C^1$ in the interval $(r,R)$ and satisfies  the following equality, in the classical sense,
$$
(s^{N-1}(|u'_0(s)|^{p(s)-2}u(s)+a_0|u'_0(s)|^{p(s)-2}u(s)))'=s^{N-1}|u_0(s)|^{q(s)-2}u_0(s), \quad s \in (r,R).
$$
It follows that for all $\psi \in E_{0}(A_{R,r})$
\begin{equation} \label{B1}
\int_{A_{R,r}}(|\nabla u_0|^{p(x)-2}\nabla u_0\nabla \psi +a(x)|\nabla u_0|^{m(x)-2}\nabla u\nabla \psi)\,dx-\int_{A_{R,r}}|u_0|^{q(x)-2}u_0 \psi\,dx=0,
\end{equation}
where $E_{0}(A_{R,r})=\{u \in E\,:\, u=0 \quad \mbox{on} \quad \partial (A_{R,r})\}$.

Using the above information, we are ready to prove that
$$
\int_{\Omega}(|\nabla u_0|^{p(x)-2}\nabla u_0\nabla v+a(x)|\nabla u_0|^{m(x)-2}\nabla u\nabla v)\,dx-\int_{\Omega}|u_0|^{q(x)-2}u_0v\,dx=0, \quad \forall v \in E.
$$
In what follows, we consider an even function $\phi \in C^{\infty}(\mathbb{R},\mathbb{R})$ satisfying
$$
0 \leq \phi(s) \leq 1 \quad \forall s \in \mathbb{R}, \quad \phi(s)=0 \quad \forall s\in [-1,1] \quad \mbox{and} \quad  \phi(s)=1 \quad \forall s\in (-\infty,-2] \cup [2,+\infty).
$$

For $\epsilon>0$ small enough and $v \in C_0^{\infty}(B_R) \subset  E_0(B_R)=\{u \in E\,:\, u=0 \quad \mbox{on} \quad \partial (B_R)\}$, we set the function
$$
v_\epsilon(x)=\phi((|x|-r)/\epsilon)v(x), \quad \forall x \in \Omega.
$$
From the definition of $v_\epsilon$, we have that $v_\epsilon \in E_0(B_r(0)) \subset X$ and $v_\epsilon \in E_0(A_{R,r})$. Thus, by (\ref{B0}) and (\ref{B1}),
$$
\int_{B_r(0)}(|\nabla u_0|^{p(x)-2}\nabla u_0\nabla v_\epsilon+a(x)|\nabla u_0|^{m(x)-2}\nabla u\nabla v _\epsilon)\,dx-\int_{B_r(0)}|u_0|^{q(x)-2}u_0v_\epsilon\,dx=0
$$
and
$$
\int_{A_{R,r}}(|\nabla u_0|^{p(x)-2}\nabla u_0\nabla v_\epsilon+a(x)|\nabla u_0|^{m(x)-2}\nabla u\nabla v _\epsilon)\,dx-\int_{A_{R,r}}|u_0|^{q(x)-2}u_0v_\epsilon\,dx=0,
$$
leading to
$$
\int_{\Omega}(|\nabla u_0|^{p(x)-2}\nabla u_0\nabla v_\epsilon++a(x)|\nabla u_0|^{m(x)-2}\nabla u\nabla v _\epsilon)\,dx-\int_{\Omega}|u_0|^{q(x)-2}u_0v_\epsilon\,dx=0
$$
or equivalently
\begin{equation} \label{E2}
\begin{array}{l}
\displaystyle \int_{\Omega}|\nabla u_0|^{p(x)-2}\nabla u_0\nabla v \phi_{\epsilon}\,dx+\int_{\Omega}|\nabla u_0|^{p(x)-2}\nabla u_0\nabla  \phi_{\epsilon}v \,dx + \\
\mbox{} \\
 \displaystyle \int_{\Omega}|\nabla u_0|^{p(x)-2}\nabla u_0\nabla v \phi_{\epsilon}\,dx+\int_{\Omega}a(x)|\nabla u_0|^{m(x)-2}\nabla u_0\nabla  \phi_{\epsilon}v \,dx \\ \mbox{}\\
-\displaystyle \int_{B_R}|u_0|^{q(x)-2}u_0v\phi_\epsilon\,dx=0,
\end{array}
\end{equation}
where
$$
\phi_\epsilon(x)=\phi((|x|-r))/\epsilon), \quad \forall x \in \mathbb{R}^N.
$$
The Lebesgue dominated convergence theorem ensures that
\begin{equation} \label{E3}
\lim_{\epsilon \to 0}\int_{\Omega}|\nabla u_0|^{p(x)-2}\nabla u_0\nabla v \phi_{\epsilon}\,dx=\int_{\Omega}|\nabla u_0|^{p(x)-2}\nabla u_0\nabla v \,dx
\end{equation}
\begin{equation} \label{E3*}
\lim_{\epsilon \to 0}\int_{\Omega}a(x)|\nabla u_0|^{m(x)-2}\nabla u_0\nabla v \phi_{\epsilon}\,dx=\int_{\Omega}a(x)|\nabla u_0|^{m(x)-2}\nabla u_0\nabla v \,dx
\end{equation}
and
\begin{equation} \label{E4}
\lim_{\epsilon \to 0}\int_{\Omega}|u_0|^{q(x)-2}u_0 v \phi_{\epsilon}\,dx=\int_{\Omega}|u_0|^{q(x)-2} u_0 v \,dx.
\end{equation}
On the other hand, it is very important to observe that
$$
\int_{\Omega}|\nabla u_0|^{p(x)-2}\nabla u_0\nabla\phi_{\epsilon} v\,dx=\int_{A_{\epsilon}}|\nabla u_0|^{p(x)-2}\nabla u_0\nabla \phi_{\epsilon} v\,dx
$$
where
$$
A_{\epsilon}=\{x \in \Omega\,:\, ||x|-r|\leq 2 \epsilon \}.
$$
Thus, by the Lebesgue dominated convergence theorem,
$$
\lim_{\epsilon \to 0}\int_{A_{\epsilon}}|\nabla u_0|^{p(x)}\,dx=0,
$$
and so,
$$
\lim_{\epsilon \to 0}\|\nabla u_0\|_{p(x),A_{\epsilon}}=0.
$$
On the other hand, since $v \in C^{\infty}_0(B_R)$, we have that
$$
\int_{ \mathbb R^N }|\nabla \phi_{\epsilon}(x)|^{p_+}|v(x)|^{p_+}\,dx \leq |v|_{\infty}^{p_+}\int_{ \mathbb R^N }|\nabla \phi_{\epsilon}(x)|^{p_+}\,dx\leq |v|_{\infty}^{p_+}\epsilon^{N-p_+}\int_{A_\epsilon}\left|\phi'\left(|z|-\frac{r}{\epsilon}\right)\right|^{p_+}\,dz
$$
then
$$
\int_{ \mathbb R^N }|\nabla \phi_{\epsilon}(x)|^{p_+}|v(x)|^{p_+}\,dx \leq C_N |\phi'|_{\infty}^{p_+}|v|_{\infty}^{p_+}\epsilon^{N-p_+}((r+2\epsilon)^N-(r-2\epsilon)^N) \to 0 \quad \mbox{as} \quad \epsilon \to 0.
$$
A similar argument works to prove that
$$
\int_{ \mathbb R^N }|\nabla \phi_{\epsilon}(x)|^{p_-}|v(x)|^{p_-}\,dx \to 0 \quad \mbox{as} \quad \epsilon \to 0.
$$
From this,
$$
\||\nabla \phi_{\epsilon}||v|\|_{p(x)} \to 0 \quad \mbox{as} \quad \epsilon \to 0.
$$
Hence, by using again H\"older's inequality, we get
\begin{equation} \label{E5}
\lim_{\epsilon \to 0}\int_{\Omega}|\nabla u_0|^{p(x)-2}\nabla u_0\nabla  \phi_{\epsilon}v \,dx=0.
\end{equation}
A similar argument gives
\begin{equation} \label{E5*}
\lim_{\epsilon \to 0}\int_{\Omega}a(x)|\nabla u_0|^{m(x)-2}\nabla u_0\nabla  \phi_{\epsilon}v \,dx=0.
\end{equation}

Taking the limit of $\epsilon \to 0$ in (\ref{E2}) and using (\ref{E3})-(\ref{E5*}), we obtain
for all $v \in C^{\infty}_0(B_R)$
\begin{equation*} \label{E06}
\int_{\Omega}(|\nabla u_0|^{p(x)-2}\nabla u_0\nabla v+a(x)|\nabla u_0|^{m(x)-2}\nabla u_0\nabla v)\,dx-\int_{\Omega}|u_0|^{q(x)-2}u_0v\,dx=0,
\end{equation*}
then by density
\begin{equation} \label{E6}
\int_{\Omega}(|\nabla u_0|^{p(x)-2}\nabla u_0\nabla v+a(x)|\nabla u_0|^{m(x)-2}\nabla u_0\nabla v)\,dx-\int_{\Omega}|u_0|^{q(x)-2}u_0v\,dx=0, \quad \forall v \in E_0(B_R).
\end{equation}

Now, we show that the above equality holds for any $w \in C_0^{\infty}(\Omega) \subset E$. The idea is as above, we set the function
$$
w_\epsilon(x)=\phi((|x|-R)/\epsilon)w(x), \quad \forall x \in \Omega,
$$
which belongs to $E_0(B_R)$ and $E_0(\Omega \setminus \overline{B}_{R}) \subset E$. Since $w_{\epsilon}|_{\Omega \setminus \overline{B}_{R}} \in E$, it follows that
$$
\int_{\Omega \setminus \overline{B}_{R}}(|\nabla u_0|^{p(x)-2}\nabla u_0\nabla w_\epsilon+|\nabla u_0|^{m(x)-2}\nabla u_0\nabla w_\epsilon) \,dx-\int_{\Omega \setminus \overline{B}_{R}}|u_0|^{q(x)-2}u_0w_{\epsilon}\,dx=0.
$$

On the other hand, as $w_{\epsilon}|_{B_{R}} \in E_0(B_R)$, by (\ref{E6}),
$$
\int_{B_{R}}(|\nabla u_0|^{p(x)-2}\nabla u_0\nabla w_\epsilon+a(x)|\nabla u_0|^{m(x)-2}\nabla u_0\nabla w_\epsilon) \,dx-\int_{B_{R}}|u_0|^{q(x)-2}u_0w_{\epsilon}\,dx=0.
$$

Combining the last two equalities we obtain
\begin{equation} \label{E7}
\int_{\Omega}(|\nabla u_0|^{p(x)-2}\nabla u_0\nabla w_\epsilon+a(x)|\nabla u_0|^{m(x)-2}\nabla u_0\nabla w_\epsilon) \,dx-\int_{\Omega}|u_0|^{q(x)-2}u_0w_{\epsilon}\,dx=0.
\end{equation}

Now, the same argument used for function $v_{\epsilon}$ works  to conclude that taking the limit  $\epsilon \to 0$ in (\ref{E7}) we obtain
$$
\int_{\Omega}(|\nabla u_0|^{p(x)-2}\nabla u_0\nabla w+a(x)|\nabla u_0|^{m(x)-2}\nabla u_0\nabla w ) \,dx-\int_{\Omega}|u_0|^{q(x)-2}u_0w\,dx=0,\quad \forall w \in C_0^{\infty}(\Omega).
$$
Again by density, we have that
$$
\int_{\Omega}(|\nabla u_0|^{p(x)-2}\nabla u_0\nabla w+a(x)|\nabla u_0|^{m(x)-2}\nabla u_0\nabla w ) \,dx-\int_{\Omega}|u_0|^{q(x)-2}u_0w\,dx=0,\quad \forall w \in E,
$$
showing that $u_0$ is a nontrivial solution of $(P_2)$. \qed

\section{ Problem 3: The case where $q$ vanishes close to the boundary }

In this section, we  study the existence of nontrivial solutions for the following class of problems:
$$
\left\{
\begin{array}{l}
-\Delta_{p(x)}u-\Delta_{m(x),a(x)}u=\lambda|u|^{q(x)-2}u \quad \mbox{in} \quad \Omega, \\
u=0 \quad \mbox{on} \quad \partial \Omega,
\end{array}
\right.
\eqno{(P_3)}
$$
where $\lambda>0$ is a parameter and $\Omega$ is a bounded domain with smooth boundary. We assume that there exist positive numbers $r<R$ such that $B_R(0) \subset \Omega$,
$$
A_{R,r} \subset \Omega_{\delta} \quad \mbox{and} \quad a(x)=a_0 \quad \forall x \in A_{R,r},
$$
where
$$
\Omega_\delta=\{x \in \Omega  \,:\, {\rm dist}\,(x,\partial \Omega) > \delta \}.
$$

Related to the functions $p,q,m,a:\overline{\Omega} \to \mathbb{R}$, we assume that they are continuous and satisfy the following conditions:
$$
\left\{
\begin{array}{l}
1<p_-=\min_{x \in \overline{\Omega} }p(x) \leq \max_{x \in \overline{\Omega}}p(x)=p_+<N, \\
1<m_-=\min_{x \in \overline{\Omega} }m(x) \leq \max_{x \in \overline{\Omega} }m(x)=m_+<N.
\end{array}
\right.
\eqno(H9)
$$
$$
\max\{p_+,m_+\}<q^{A}_{-}=\min_{x \in \overline{\Omega}_\delta \setminus A_{R,r} }q(x)\leq q^{A}_{+}=\max_{x \in \overline{\Omega}_\delta \setminus A_{R,r} }q(x) < \min_{x \in \overline{\Omega}}p^*(x).
\eqno(H10)
$$
$$
0 \leq a(x) \leq L, \quad \forall x \in \overline{\Omega} \eqno(H11)
$$
and
$$
q(x) \geq 0 \quad \forall x \in \overline{\Omega} \quad \mbox{and} \quad \lim_{ {\rm dist}\,(x,\partial \Omega) \to 0}q(x)=0. \eqno(H12)
$$

\begin{theorem}\label{t3}
Assume that hypotheses $(H9)-(H12)$ are fulfilled. Then
there exists $\lambda^*>0$ such that for all $\lambda \in (0, \lambda^*)$
problem $(P_3)$
has at least two nontrivial solutions in $E$.
\end{theorem}

As in the previous section, the natural candidate to be the energy associated to problem $(P_3)$ is  the following double-phase functional with variable exponents:
$$
I(u)=\int_{\Omega}\left(\frac{1}{p(x)}|\nabla u|^{p(x)}+\frac{a(x)}{m(x)}|\nabla u|^{m(x)}\right)\,dx-\int_{\Omega}\frac{\lambda}{q(x)}|u|^{q(x)}\,dx.
$$
Since we do not assume any growth condition on $q$ in the annulus $A_{R,r}$, this functional
 is not well defined in the whole space $E$. That is why we restrict the functional $I$ to the closed subspace $X \subset E $ given by
$$
X=\{u \in E \,:\,u(x)=u(|x|) \quad \mbox{a.e.} \quad x \in \overline{A}_{R,r} \}.
$$
Arguing as in Section 2, we deduce that the compact embedding (\ref{EB1}) still holds in the present case,  hence $I \in C^{1}(X,\mathbb{R})$ and
$$
I'(u)v=\int_{\Omega}(|\nabla u|^{p(x)-2}\nabla u\nabla v+a(x)|\nabla u|^{m(x)-2}\nabla u\nabla v)\,dx-\lambda \int_{\Omega}|u|^{q(x)-2}uv\,dx, \quad \forall u,v \in X.
$$
Moreover, $I$ also satisfies the $(PS)$ condition in $X$.

\begin{lemma} \label{LEMA1} Given $\tau>0$, there are $\rho=\rho(\tau)>0$ and $\lambda^*=\lambda^*(\tau)$ such that
$$
I_\lambda(u)\geq \rho \quad \mbox{for} \quad \|u\|=\tau \quad \mbox{and} \quad \lambda \in (0,\lambda^*).
$$	
\end{lemma}
\begin{proof} By our assumptions we know that there exists $C>0$ such that
$$
\int_{\Omega}|u|^{q(x)}\,dx \leq C\max\{\|u\|^{q_+},\|u\|^{q_-}\}.
$$	
From this, for  $\|u\|=\tau$,
$$
\int_{\Omega}|u|^{q(x)}\,dx \leq C\max\{\tau^{q_+},\tau^{q_-}\}=C_\tau,
$$	
and so,
$$
I_\lambda(u) \geq \min\{\tau^{p_+},\tau^{p_-}\}-\lambda C_\tau, \quad \|u\|=\tau.
$$
Setting $\lambda^*=\frac{\min\{\tau^{p_+},\tau^{p_-}\}}{2C_\tau}$ and $\lambda \in (0,\lambda^*)$, we get
$$
I_\lambda(u) \geq \frac{\min\{\tau^{p_+},\tau^{p_-}\}}{2}=\rho(\tau)>0, \quad \|u\|=\tau.
$$
The proof of this auxiliary result is now concluded.
\end{proof}

\begin{lemma} \label{LEMA2}
Setting $A_\lambda=\inf\{I_\lambda(u)\,:\,\|u\| \leq \tau\}$, we have that $A_\lambda<0$ for all $\lambda \in (0,\lambda^*)$. 	
\end{lemma}
\begin{proof} Let $\phi \in C_0^{\infty}(B_\alpha(x_0))$ such that
$$
q_1=\max_{x \in \overline{B}_{\alpha}(x_0)}q(x)<\min\{p_-,m_-\}.
$$	
Here, we are using the fact that $\displaystyle \lim_{ dist(x,\partial \Omega) \to 0}q(x)=0.$ If $t>0$ is small enough, a simple computation gives
$$
I_{\lambda}(t\phi) \leq \int_{\Omega}\left(\frac{t^{p_-}}{p(x)}|\nabla \phi|^{p(x)}+\frac{t^{m_-}a(x)}{m(x)}|\nabla \phi|^{m(x)}\right)\,dx-t^{q_1}\int_{\Omega}\frac{\lambda}{q(x)}|\phi|^{q(x)}\,dx.
$$
It follows that
$$
I_{\lambda}(t\phi)<0 \quad \mbox{for} \quad t\approx 0^+,
$$	
showing the desired result.
\end{proof}

The last two lemmas permit to apply the Ekeland variational principle to conclude that  there exists $u_\lambda \in X$ such that
$$
I'_\lambda(u_\lambda)v=0, \quad \forall v \in X  \quad \mbox{and} \quad I_\lambda(u_\lambda)=A_\lambda<0.
$$
Repeating the same arguments of the last section, it follows that $u_\lambda$ is a critical point of $I_\lambda$ in $E$ for all $\lambda \in (0, \lambda^*)$.

\begin{lemma} \label{LEMA3} Fixed $\phi \in C_0^{\infty}({\Omega}_\delta \setminus \overline{A}_{R,r})$, we have that
$$
I_\lambda(t \phi) \to -\infty \quad \mbox{as} \quad t \to +\infty.
$$	
\end{lemma}
\begin{proof} If  $\phi \in C_0^{\infty}({\Omega}_\delta \setminus \overline{A}_{R,r})$ and $t>0$ is large enough, by $(H10)$  we have that
	$$
	I_{\lambda}(t\phi) \leq \int_{\Omega}\left(\frac{t^{p_+}}{p(x)}|\nabla \phi|^{p(x)}+\frac{t^{m_+}a(x)}{m(x)}|\nabla \phi|^{m(x)}\right)\,dx-t^{q^{A}_{-}}\int_{\Omega}\frac{\lambda}{q(x)}|\phi|^{q(x)}\,dx.
	$$
Hence
	$$
\lim_{t \to +\infty}I_{\lambda}(t\phi)=-\infty,
	$$	
showing the desired result.
\end{proof}

\medskip
\noindent {\it Proof of Theorem \ref{t3}.} \, From Lemmas \ref{LEMA1} and \ref{LEMA3},  we derive that $I_\lambda$ satisfies the mountain pass geometry. Applying \cite[Theorem 1.1]{J1} (see also \cite{J2}), we obtain that for almost every $ \lambda \in (0,\lambda^*)$ there is a bounded  $(PS)_{c_\lambda}$ sequence for $I_\lambda$, where $c_\lambda$ is the mountain level of $I_\lambda$. Since  $I_\lambda$ verifies the $(PS)$ condition, it follows that for almost every $ \lambda \in (0,\lambda^*)$ the level $c_\lambda$ is a critical level, that is, there is $u^{\lambda} \in X$ such that
$$
I'_\lambda(u^\lambda)=0 \quad \mbox{and} \quad I_\lambda(u^\lambda)=c_\lambda>0.
$$
We conclude that problem $(P_3)$ has at least two solutions $u_\lambda$ and $u^{\lambda}$ for almost every $\lambda \in (0, \lambda^*)$ with
$$
I_{\lambda}(u_\lambda)=A_{\lambda}<0 \quad \mbox{and} \quad I_{\lambda}(u^\lambda)=c_{\lambda}>0.
$$

Finally, we can repeat the same arguments developed in Section 3 to conclude that $u_\lambda$ and $u^{\lambda}$ are, in  fact, critical points of $I_\lambda$ in $E$, hence two nontrivial solutions of problem $(P_3)$. \qed

\end{document}